\documentclass{amsart}[12pt]
\usepackage{amssymb,amsmath}
\usepackage{enumerate}
\usepackage[english]{babel}
\usepackage{color}

\newcommand{\Om} {\Omega}
\newcommand {\ep} {\varepsilon}
\newcommand {\om} {\omega}

\newcommand {\sg} {\sigma}
\newcommand {\ii} {\infty}
\newcommand {\dt} {\delta}
\newcommand {\al} {\alpha}
\newcommand {\bt} {\beta}
\newcommand {\lb} {\lambda}
\newcommand {\Lb} {\Lambda}
\newcommand {\la} {\longrightarrow}

\newcommand {\sm} {\setminus}
\newcommand {\su} {\subset}
\newcommand {\wt} {\widetilde}
\newcommand {\wh} {\widehat}
\newcommand {\mb} {\mathbf}
\newcommand {\mbb} {\mathbb}
\newcommand {\pr} {\prime}
\newcommand {\mc} {\mathcal}

\newtheorem{teo}{Theorem}[section]
\newtheorem{pro}{Proposition}[section]
\newtheorem{cor}{Corollary}[section]

\newtheorem{lm}{Lemma}[section]

\theoremstyle{definition}
\newtheorem{rem}{Remark}[section]

\title{Besicovitch-weighted \\ergodic theorems with continuous time}
\keywords{Sigma-finite measure, Dunford-Schwartz operator, continuous semigroup, bounded Besicovitch function, almost uniform convergence, fully symmetric space}
\subjclass[2010]{47A35(primary), 46L52(secondary)}
\begin{document}
\date{January 12, 2025}

\begin{abstract}
Given $1\leq p<\ii$, we show that ergodic flows in the $L^p$-space over a $\sg$-finite measure space generated by strongly continuous semigroups of Dunford-Schwartz operators and modulated by bounded Besicovitch almost periodic functions converge almost uniformly (in Egorov's sense). The corresponding local ergodic theorem is proved with identification of the limit. Then we extend these results to arbitrary fully symmetric spaces, including Orlicz, Lorentz, and Marcinkiewicz spaces. 
\end{abstract}

\author{SEMYON LITVINOV}
\address{Pennsylvania State University, 76 University Drive, Hazleton, PA 18202, USA}
\email{snl2@psu.edu}

\maketitle
\section{Introduction}
Since the groundbraking work \cite{ry}, pointwise convergence of Besicovitch-weighted ergodic averages has been examined quite extensively; see, for example, \cite{bo, bl, lot, dj, ccl}. However, all these studies were concerned with discrete time, thus, there is a natural question what happens in the realm of continuous time. In this case, of course, one must replace Besicovitch sequences by the well-known Besicovitch almost periodic functions \cite{be}. 

Initially, this problem was studied in the non-commutative setting\,--\,for actions of strongly continuous semigroups of {\it positive} Dunford-Schwartz operators\,--\,in \cite{mk} for $t\to0^+$ and in \cite{cl20} for $t\to0^+$ and for $t\to\ii$. In both articles, non-commutative counterparts of the almost uniform convergence (in Egorov's sense) were employed. 

Back to the commutative case with infinite measure, it is worth asking whether {\it almost uniform convergence}\,--\,which is generally stronger than almost everywhere convergence\,--\,of Besicovitch-weighted ergodic averages with continuous time will hold for actions of strongly continuous semigroups of Dunford-Schwartz operators that are {\it not presumed positive}. Besides, as the measure is not finite, there is another question: what happens within the the universal $(L^1+L^\ii)$\,-\,space {\it outside the $L^1$-\,space}. These are the three questions we aim to address in this article.

Let $(\Om,\mu)$ be a $\sg$-finite measure space. Given $1\leq p\leq\ii$, denote $\mc L^p=\mc L^p(\Om,\mu)$. Let $\mc T=\{T_s\}_{s\ge 0}$ be a semigroup of Dunford-Schwartz operators in $\mc L^1+\mc L^\ii$ which is strongly continuous in $\mc L^1$. In Section 4 of the article, we show that the corresponding ergodic averages $\{M_t(\mc T)(f)\}_{t>0}$ converge almost uniformly as $t\to\ii$ and as $t\to0^+$ for every $f\in\mc L^p$, where $1\leq p<\ii$. Note that the pointwise, that is almost everywhere, convergence of these avrages was studied in \cite{te} and later in \cite{mg} and \cite{sa}.

Then we proceed in Section 5 by showing that almost uniform convergence in the space $\mc L^p$, where $1\leq p<\ii$, holds when the above ergodic averages are modulated by the bounded Besicovitch (when $t\to\ii$) or locally Besicovitch (when $t\to0^+$) almost periodic functions. In the case $t\to0^+$, the limits are identified.

In Section 6, we extend these results to arbitrary fully symmetric spaces. Such include the well-known Orlicz, Lorentz, and Marcinkiewicz spaces.

\section{Preliminaries}
Let $(\Om,\mc A,\mu) $ be a $\sg$-finite infinite measure space and let $\mc L^0 =\mc L^0(\Om,\mu)$ be the $*$-algebra of equivalence classes of almost everywhere (a.e.) finite complex-valued measurable functions on $\Om$. Given $1\leq p\leq\ii$, let $\mc L^p\su\mc L^0$ be the $\mc L^p$-space on $\Om$ equipped with the standard Banach norm $\|\cdot\|_p$.

A net $\{f_\al\}\su\mc L^0$ is said to converge {\it almost uniformly (a.u.)} to $f\in\mc L^0$ (in Egorov's sense) if for every $\ep>0$ there exists a set $G\su\Om$ such that $\mu(\Om\sm G)\leq\ep$ and $\|(f-f_\al)\chi_G\|_\ii\to 0$, where $\chi_G$ is the  characteristic function of  set $G$. It is clear that every a.u. convergent net converges almost everywhere (a.e.) and that the converse is not true in general.

Note that a.u. convergence may not agree with any topology in $\mc L^0$ -- see, for example, \cite{or}. A proof of the following statement can be found in \cite[Theorem 3.2]{li1}. Although it is given there in the non-commutative setting, one only needs to adjust notation and then verbatim repeat the argument.

\begin{pro}\label{p0}
The space $\mc L^0$ is complete with respect to a.u. convergence
\end{pro}

Traditionally, a linear operator\, $T: \mc L^1\to  \mc L^1$ is called a {\it Dunford-Schwartz operator} (in short, $T \in DS$) if
\[
\| T(f)\|_1\leq \| f\|_1 \ \ \forall \ f\in \mc L^1 \text{ \ and \ } \| T(f)\|_{\ii}\leq \| f\|_{\ii} \ \ \forall \ f\in \mc L^\ii\cap\mc L^1;
\]
see \cite[Ch.\,VIII,\,\S\,6]{ds}, \cite{ga}, \cite[\S\S\,4.1,\,4.2]{kr}. 

\vskip5pt
\noindent
However, as it was proved in \cite[Theorem 3.3]{ccl} (see also \cite[p.\,42]{sa}), one can assume without loss of generality that $T\in DS$ contracts the entire space $\mc L^\ii$:

\begin{teo}\label{t0}
For any Dunford-Schwartz operator $T$ there exists a unique linear operator\, $\wt T: \mc L^1+\mc L^\ii\to  \mc L^1+\mc L^\ii$ such that
\[
\wt T(f)=T(f)\  \ \forall \ f\in\mc L^1, \ \ \|\wt T(f)\|_\ii\leq\|f\|_\ii\ \, \forall \ f\in\mc L^\ii,
\]
and\, $\wt T|_{\mc L^\ii}$ is\, $\sg(\mc L^\ii,\mc L^1)$-continuous.
\end{teo}
\noindent
This observation is of the essence for us as we aim to extend the results initially developed for $\mc L^1$ to arbitrary fully symmetric spaces $E\su\mc L^1+\mc L^\ii$. 
\vskip5pt
If a Dunford-Schwartz operator $T$ is positive, that is, $T(f)\ge 0$ whenever $f\ge 0$, we will write $T\in DS^+$.

\vskip5pt
Let $X$ be a Banach space, and let $M_\al: X\to\mc L^0$, $\al\in\Lambda$, be a family of linear maps. The function on $X$ given by
\[
M^*(f)=\sup_{\al\in\Lambda}|M_\al(f)|, \ \, f\in X,
\]
is called the {\it maximal function} of the family $\{M_\al\}_{\al\in\Lambda}$. If $M^*(f)\in\mc L^0$ for all $f\in X$, then the map $M^*:X\to\mc L^0$ is colled the {\it maximal operator} of $\{M_\al\}_{\al\in\Lambda}$.

Given $f\in \mc L^1+\mc L^\ii$ and $T\in DS$, denote
\[
M_n(f)=\frac1n\sum_{k=0}^{n-1}T^k(f), \ n=1,2,\dots\text{ \ \ and\ \ }M(T)^*(f)=\sup_n|M_n(f)|.
\]

For a proof of the following {\it maximal ergodic inequality} we refer the reader to \cite[Theorem 3.3]{cl19}.

\begin{teo}\label{t1}
If \,$T\in DS$ and $1\leq p<\ii$, then 
\[
\mu\big\{M(T)^*(|f|)>\lb\big\}\leq\left(2\frac{\|f\|_p}\lb\right)^p \ \ \, \forall \ f\in\mc L^p,\ \lb>0.
\]
\end{teo}

In what follows, $\mc T=\{T_s\}_{s\ge 0}\su DS$ is a semigroup, {\it strongly continuous in} $\mc L^1$, that is,
\[
\| T_s(f)-T_{s_0}(f)\|_1\to 0 \text{ \ \ whenever\ \ } s\to s_0 \text{\ \ for all\ \ } f\in \mc L^1.
\]

It is well-known that if a linear operator contracts both $\mc L^1$ and $\mc L^\ii$, it also contracts each space\, $\mc L^p$ for $1<p<\ii$. In addition, utilizing the idea of the proof of Proposition 2.1 in \cite{cl19a}, we establish the following fact:

\begin{pro}\label{p8}
If a semigroup $\{T_s\}_{s\ge0}\su DS$ is strongly continuous in $\mc L^1$, then it is strongly continuous in each\, $\mc L^p$ for $1<p<\ii$.
\end{pro}
\begin{proof}
Clearly, it is sufficient to verify that
\[
\big\|T_{s_n}(f)-T_{s_0}(f)\big\|_p\to0\text{ \ \ as\,\ }s_n\to s_0 \ \ \forall \ f\in\mc L^p.
\]
So, let $f\in\mc L^p$ and, given $m\in\mbb N$, denote
\[
f_m=f\,\chi_{\left\{|f|\leq m^{-1}\right\}}+f\,\chi_{\left\{|f|\ge m\right\}}\text{ \ \ and \ \ }g_m=f-f_m.
\]
Fix $\ep>0$ and define $m_0$ to be such that $\|f_{m_0}\|_p\leq\ep/4$. As $T_s$ is a contraction in $\mc L^p$ for each $s\ge0$, we have
\[
\big\|T_{s_n}(f_{m_0})-T_{s_0}(f_{m_0})\big\|_p\leq\frac\ep2\, \ \ \forall \ n\in\mbb N.
\]

Next, note that $g_{m_0}\in\mc L^1\cap\mc L^\ii$. Since we can assume that $f\neq0$, we can also assume, without loss of generality, that $g_{m_0}\neq0$. Then, if we let
\[
h_n=\frac12\|g_{m_0}\|_\ii^{-1}\big(T_{s_n}(g_{m_0})-T_{s_0}(g_{m_0})\big),
\]
it follows, by the assumption, that $\|h_n\|_1\to0$ as $n\to\ii$. Besides, since $T_s$ is a contraction in $\mc L^\ii$ for each $s\ge0$, we have $-\mb1\leq h_n\leq\mb 1$, hence $|h_n|^p\leq|h_n|$, for each $n$, implying that
\[
\|h_n\|_p^p=\int_{\Om}|h_n|^pd\mu\leq\int_{\Om}|h_n|d\mu=\|h_n\|_1\to0\text{\, \ as\,\ }n\to\ii.
\]
Thus, there exists $n_0\in\mbb N$ such that
\[
\big\|T_{s_n}(g_{m_0})-T_{s_0}(g_{m_0})\big\|_p\leq\frac\ep2\, \ \ \forall \ n\ge n_0.
\]
Therefore, if $n\ge n_0$, it follows that
\[
\big\|T_{s_n}(f)-T_{s_0}(f)\big\|_p\leq\big\|T_{s_n}(f_{m_0})-T_{s_0}(f_{m_0})\big\|_p+\big\|T_{s_n}(g_{m_0})-T_{s_0}(g_{m_0})\big\|_p\leq\ep,
\]
and the proof is complete.
\end{proof}

\vskip5pt
The following proposition is to justify the existence of ergodic averages we are going to examine in this article. Note that the averages $M_t(\mc T)(f)$ below were actually introduced in \cite{mg} and \cite{sa} but in a less straightforward manner, based on a non-trivial application of Fubini's theorem to the product measure space $\mbb R_+\otimes\Om$.
\begin{pro}\label{p1}
Let $\nu$ be Lebesgue measure on $\mbb R_+$, and  let $\bt: \mbb R_+\to\mbb C$ be a $\nu$-measurable function with $\|\bt\|_\ii<\ii$. Given $1\leq p<\ii$ and $f\in\mc L^p$, the function $\bt(s)T_s(f)$ acting on the measure space $(\mc R_+,\nu)$ with the values in the Banach space $(\mc L^p, \|\cdot\|_p)$ is Bochner $\nu$-integrable on $[0,t]$ for every $t>0$. Therefore, the ergodic averages
\[
M_t(\mc T,\bt)(f):=\frac1t\int_0^t\bt(s)T_s(f)ds\in \mc L^p.
\]
are defined for all $f\in\mc L^p$ and $t>0$.
\vskip5pt
\noindent
In particular, the averages
\[
M_t(\mc T)(f):=\frac1t\int_0^tT_s(f)ds\in \mc L^p
\]
are also defined for all $f\in\mc L^p$ and $t>0$.
\end{pro}
\begin{proof}
Fix $f\in\mc L^p$ and let $q$ be such that $p^{-1}+q^{-1}=1$. Then, by virtue of Proposition \ref{p8} and H\"older's inequality, given $g\in\mc L^q$, the function
\[
\mbb R_+\ni s\longmapsto \int_{\Om}T_s(f)gd\mu
\]
is continuous on $\mbb R_+$, implying that the map
\[
S_f:\mbb R_+\la\mc L^p\text{ \ given by \ } S_f(s)=T_s(f)
\]
is {\it weakly $\nu$-measurable} as of \cite[Ch.\,V,\,\S\,4]{yo}. Since $S_f(\mc R_+)$ is a separable subset of the Banach space $(\mc L^p, \|\cdot\|_p)$, Pettis theorem \cite[Ch.\,V,\,\S\,4]{yo} implies that $S_f$ is {\it strongly $\nu$-measurable}, that is, there exists a sequence $U_n: \mbb R_+\to\mc L^p$ of simple maps, as in \cite[Ch.\,V,\,\S\,4]{yo}, such that
\[
\lim_{n\to\ii}\|T_s(f)-U_n(s)\|_p=0, \text{\ hence\,\ }\|T_s(f)\|_p=\lim_{n\to\ii}\|U_n(s)\|_p,\, \ \nu\text{-a.e.}
\] 
Therefore, as for any $n$ the function\ $\mbb R_+\ni s\longmapsto\|U_n(s)\|_p$ is trivially $\nu$-measurable, it follows that the function 
\[
\mbb R_+\ni s\longmapsto \|T_s(f)\|_p
\]  
is also $\nu$-measurable. Since, in addition,
\[
\|T_s(f)\|_p\leq\|f\|_p \ \, \forall \ s\in\mbb R_+,
\]
the function $\|T_s(f)\|_p$, as well as the function 
\[
\|\bt(s)T_s(f)\|_p=|\bt(s)|\,\|T_s(f)\|_p,
\]
is $\nu$-integrable on $[0,t]$, for each $t>0$.

Therefore, by \cite[Ch.\,V,\, \S\,5,\, Theorem 1]{yo}, the function $\bt(s)T_s(f)$ is Bochner $\nu$-intergable on $[0,t]$ for any given $t>0$, which completes the proof.
\end{proof}

Given $1\leq p<\ii$, we can now define the maximal function of the family $\{M_t(\mc T)\}_{t>0}$ on $\mc L^p$ as
\[
M(\mc T)^*(f)=\sup_{t>0}\big|M_t(\mc T)(f)\big|, \ \, f\in\mc L^p.
\]

We will need the following, stronger version, of Riesz convergence theorem:

\begin{pro}\label{p2}
If\, $f_n\to f$ in measure, then there exists a subsequence $\{n_k\}$ such that $f_{n_k}\to f$ a.u.
\end{pro}
\begin{proof}

Fix $\ep>0$, Since, for every given $k\in\mbb N$, we have
\[
\mu\big\{|f_n-f|>1/k\big\}\to 0,
\]
there exists $n_k$ such that
\[
\mu\big\{|f_{n_k}-f|>1/k\big\}\leq\frac\ep{2^k}.
\]
Therefore, if we define
\[
\Om_k=\big\{|f_{n_k}-f|\leq1/k\big\}\text{ \ and \ } \Om_\ep=\bigcap_{k=1}^\ii\Om_k,
\]
then $\mu(\Om\sm\Om_k)\leq\displaystyle\frac\ep{2^k}$ for each $k\in\mbb N$ implies that 
$\mu(\Om\sm\Om_\ep)\leq\ep$ and
\[
\big\|\,|f_{n_k}-f|\chi_{\Om_\ep}\big\|_\ii\leq1/k,\text{ \ hence \ }\big\|\,|f_{n_k}-f|\chi_{\Om_\ep}\big\|_\ii\to 0,
\]
that is, $f_{n_k}\to f$ a.u.
\end{proof}

We shall also notice the following known fact, together with a corollary, that will be used repeatedly in the sequel. A short proof is provided, due to their special importance in this article.

\begin{pro}\label{p9}
Given $1\leq p<\ii$, the unit ball in $\mc L^p$ is sequentially closed with respect to a.e. convergence.
\end{pro}
\begin{proof}
Let $\|f_n\|_p\leq1$ for each $n$ and $f_n\to f$ a.e. Then $|f_n|^p\to|f|^p$ a.e. also, implying, by Fatou's lemma, that
\[
\int_\Om|f|^pd\mu\leq\liminf\int_\Om|f_n|^pd\mu\leq1,
\]
hence $\|f\|_p\leq1$.
\end{proof}

Since a.u. convergence implies a.e. convergence, and in view of Proposition \ref{p2}, we also have this:

\begin{cor}\label{c3}
For every $1\leq p<\ii$, the unit ball in $\mc L^p$ is sequentially closed with respect to a.u. convergence and convergence in measure.
\end{cor}

\section{Maximal ergodic inequality for the family $\{M_t(\mc T,\bt)\}_{t>0}$}

Let $\bt: \mbb R_+\to\mbb C$ is a Lebesgue measurable function with $\|\bt\|_\ii\leq C<\ii$. Denote
\[
M(\mc T,\bt)^*(f)=\sup_{t>0}|M_t(\mc T,\bt)(f)|,
\]
the maximal function of the family $\{M_t(\mc T,\bt)\}_{t>0}$.

\begin{teo}\label{t2} Let $1\leq p<\ii$. If $\mc T=\{T_s\}_{s\ge0}$ and $\bt(s)$ are as above, then
\begin{equation}\label{e3}
\mu\,\Big\{M(\mc T,\bt)^*(f)>\lb\Big\}\leq\left(\frac{4C\,\|f\|_p}\lb\right)^p
\end{equation}
and, in particular,
\begin{equation}\label{eq4}
\mu\,\Big\{M(\mc T)^*(f)>\lb\Big\}\leq\left(\frac{4\,\|f\|_p}\lb\right)^p
\end{equation}
for all $f\in \mc L^p$ and $\lb>0$.
\end{teo}
\begin{proof}
Fix $f\in\mc L^p$, $\lb>0$, $\displaystyle\frac nm\in \mbb Q$, where $n, m \in \mbb N$, and denote  
\[
g = \int_0^1\big|T_{s/m}(f)\big|ds.
\]
Let $|T|$ be the linear modulus of $T\in DS$ -- see, for example, \cite[\S\,4.1]{kr}. Then $|T|\in DS^+$ and 
\[
|T^k(f)|\leq|T|^k(|f|)\ \ \forall\ f\in\mc L^p,
\]
thus, we can write
\begin{equation*}
\begin{split}
\big|M_{n/m}(\mc T,\bt)(f)\big|&=\Big|\frac mn\int_0^{\frac nm}\bt(s)T_s(f) ds\Big| \leq2\,\frac mn\int_0^{\frac nm}\big|\bt(s)T_s(f)\big|ds\\
&=2\,\frac mn\int_0^{\frac nm}|\bt(s)|\big|T_s(f)\big|ds\leq 2C\,\frac mn\int_0^{\frac nm}\big|T_s(f)\big|ds\\
&=\frac{2C}n \int_0^{n}\big|T_{s/m}(f)\big| ds=\frac{2C}n\sum_{k=0}^{n-1}\int_k^{k+1}\big|T_{s/m}(f)\big|ds\\
&=\frac{2C}n\sum_{k=0}^{n-1}\int_0^1\big|T_{(s+k)/m}(f)\big|ds=\frac{2C}n\sum_{k=0}^{n-1}\int_0^1\big|T^k_{1/m}(T_{s/m}(f))\big|ds\\
&\leq \frac{2C}n\sum_{k=0}^{n-1}\int_0^1\big|T_{1/m}\big|^k\left(\big|T_{s/m}(f)\big|\right)ds\leq 2C\,\frac1n\sum_{k=0}^{n-1}\big|T_{1/m}\big|^k(g)\\
&=2C\,M_n\left(\big|T_{1/m}\big|\right)(g),
\end{split}
\end{equation*}
which, by Theorem \ref{t1} applied to the operator $|T_{1/m}|\in DS^+$, entails, as $g\ge 0$,
\begin{equation*}
\begin{split}
\mu\left\{\big|M_{n/m}(\mc T,\bt)(f)\big|>\lb\right\}&\leq\mu\Big\{M_n(|T_{1/m}|)(g)>\frac\lb{2C}\Big\}\\&\leq\left(\frac{4C\,\|g\|_p}\lb\right)^p\leq\left(\frac{4C\,\|f\|_p}\lb\right)^p
\end{split}
\end{equation*}
for all $m,n\in\mbb N$, hence
\[
\mu\left\{\sup_{0<r\in\mbb Q}\big|M_r(\mc T,\bt)(f)\big|>\lb\right\}\leq\left(\frac{4C\,\|f\|_p}\lb\right)^p.
\]

To finish the argument, fix $f\in\mc L^p$ and $\lb>0$ and denote
\[
f_t=\big|M_t(\mc T,\bt)(f)\big|, \ \ \ E=\Big\{\sup_{0<r\in\mbb Q}f_r>\lb\Big\}.
\]
Then we have
\[
\mu(E)\leq\left(\frac{4C\,\|f\|_p}\lb\right)^p\text{ \ and \ }\big\|f_r\chi_{\Om\sm E}\big\|_\ii\leq\lb\ \ \, \forall\ 0<r\in\mbb Q.
\]
If\, $t>0$ and $0<r_n\to t$ with $\{r_n\}\su \mbb Q$, then we have $f_{r_n}\to f_t$ in $\mc L^p$, hence in measure. Then, by Proposition \ref{p2}, there exists a subsequence $\{r_{n_k}\}$ such that $f_{r_{n_k}}\to f$ a.u., implying that
\[
\big\|f_t\chi_{\Om\sm E}\big\|_\ii\leq\lb,
\]
hence
\[
\mu\big\{f_t>\lb\big\}\leq\mu(E)\leq\left(\frac{4C\,\|f\|_p}\lb\right)^p\ \ \, \forall\ t>0,
\]
which entails the maximal inequality (\ref{e3}), hence (\ref{eq4}), for all $f\in\mc L^p$ and $\lb>0$.
\end{proof}

Now, as a corollary of Theorem \ref{t2}, we will show that the a.u.-convergence set of the net $\{M_t(\mc T,\bt)\}_{t>0}$ and, in particular, of $\{M_t(\mc T)\}_{t>0}$, as $t\to\ii$ (or as $t\to0$) is closed in the Banach space $(\mc L^p, \|\cdot\|_p)$ for each $1\leq p<\ii$. For this we will utilize the notion of {\it measure topology} in $\mc L^0$, which is given by the following system of neighborhoods of zero:
\[
\mc N(\ep,\dt)=\left\{ f\in\mc L^0:\ \mu\{|f|>\dt\}\leq\ep\right\}, \ \, \ep>0,\ \dt>0.
\]

\begin{pro}\label{p3}
Let $(X,\|\cdot\|)$ be a Banach space, and let $M_t: X\to\mc L^0$, $t>0$, be a family of lenear operators. If the maximal operator $M^*: (X,\|\cdot\|)\to(\mc L^0,t_\mu)$ is continuous at $0\in X$, then the set
\[
X_c=\big\{\,f\in X: \{M_t(f)\}\text{\ converges a.u. as}\ \,t\to\ii\text{\ (or as\ } t\to0)\big\}
\]
is closed in $X$.
\end{pro}
\begin{proof}
Assume that $t\to\ii$ in the dedinition of $X_c$. Fix $\ep>0$ and $\dt>0$ and let a sequence $\{f_k\}\su X_c$ and $f\in X$ be such that $\|f-f_k\|\to 0$. Then the continuity of $M^*$ at $0\in X$ implies there exist $f_{k_0}$ and $G_0\su\Om$ satisfying conditions
\[
\mu(\Om\sm G_0)\leq\frac\ep2\text{\ \, and\ \, } \Big\|\sup_{t>0}M_t(f-f_{k_0})\chi_{G_0}\Big\|_\ii\leq\frac\dt3.
\]

Next, since the net $\{M_t(f_{k_0})\}$ converges a.u. as $t\to\ii$, there exist $G_1\su\Om$ and $t_0>0$ such that
\[
\mu(\Om\sm G_1)\leq\frac\ep2\text{\ \, and\ \, } \big\|\big(M_{t^\pr}(f_{k_0})-M_{t^{\pr\pr}}(f_{k_0})\big)\chi_{G_1}\big\|_\ii\leq\frac\dt3\ \ \ \forall\, \ t^\pr,t^{\pr\pr}\ge t_0.
\]

Now, if we define\, $G=G_0\cap G_1$, then $\mu(\Om\sm G)\leq\ep$ and
\begin{equation*}
\begin{split}
\big\|\big(M_{t^\pr}(f)-M_{t^{\pr\pr}}(f)\big)\chi_G\big\|_\ii&\leq\big\|M_{t^\pr}(f-f_{k_0})\chi_G\big\|_\ii+\big\|M_{t^{\pr\pr}}(f-f_{k_0})\chi_G\big\|_\ii\\
&+\big\|\big(M_{t^\pr}(f_{k_0})-M_{t^{\pr\pr}}(f_{k_0})\big)\chi_G\big\|_\ii\leq\dt\, \ \ \forall\, \ t^\pr,t^{\pr\pr}\ge t_0,
\end{split}
\end{equation*}
thus, the net $\{M_t(f)\}_{t>0}$ is Cauchy as $t\to\ii$. This, in view of Proposition \ref{p0}, implies that this net converges a.u. as $t\to\ii$ in $\mc L^0$, hence $f\in X_c$. 
\vskip5pt
The case $t\to0$ is dealt with in the same manner.
\end{proof}

\begin{cor}\label{c1}
Given $1\leq p<\ii$, the a.u.-convergence set of the net $\{M_t(\mc T,\bt)\}_{t>0}$, hence of $\{M_t(\mc T)\}_{t>0}$, as $t\to\ii$ (or as $t\to0$) is closed in $\mc L^p$.
\end{cor}
\begin{proof}
Without loss of generality, assume that $C>0$. Then, given $\ep>0$ and $\dt>0$, condition $\|f\|_p\leq\displaystyle\frac{\ep^{1/p}\dt}{4C}$ implies, by Theorem \ref{t2}, that
\[
\mu\big\{M(\mc T,\bt)^*(f)>\dt\big\}\leq\ep,
\]
hence $M(\mc T,\bt)^*(f)\in\mc N(\ep,\dt)$. Therefore  $M(\mc T,\bt)^*: (\mc L^p, \|\cdot\|_p)\to (\mc L^0, t_\mu)$ is continuous at zero in $\mc L^p$, and the statement follows from Proposition \ref{p3}.
\end{proof}

\section{Almost uniform convergence of non-weighted averages in $\mc L^p(\Om,\mu)$ }
In this section we aim to establish a.u. convergence of non-weighted ergodic averages with continuous time in $\mc L^p$. For a.e. convergence of such averages, see, for example, \cite{sa}.
\vskip 5pt
We begin with the case of a {\it flow}, that is, when $t\to\ii$:
\begin{teo}\label{t3}
Given $1\leq p<\ii$ and $f\in\mc L^p$, the averages $M_t(\mc T)(f)$ converge a.u. to some $\wh f\in\mc L^p$ as $t\to\ii$.
\end{teo} 
\begin{proof}
Pick $h\in\mc L^1\cap\mc L^\ii$ and denote $g=\int_0^1T_s(h)ds$. As in the proof of Theorem \ref{t2}, one can see that
\[
M_n(\mc T)(h)=\sum_{k=0}^{n-1}T_1^k(g)=M_n(T_1)(g), \ \, n=1,2,\dots
\]
Therefore, by \cite[Teorem 3.5]{cl19}, we conclude that $M_n(\mc T)(h)\to\wh h$ a.u. for some $\wh h\in\mc L^p$. Thus, given $\ep>0$, there exists $G\su\Om$ such that $\mu(\Om\sm G)\leq\ep$ and
\[
\big\|\big(M_n(\mc T)(h)-\wh h\big)\chi_G\big\|_\ii\to 0.
\]

Next, given $t>1$, we have
\begin{equation*}
\begin{split}
\left\|M_t(\mc T)(h)-M_{[t]}(\mc T)(h)\right\|_\ii&=\left\|\frac1t\int_0^tT_s(h)ds-\frac1{[t]}\int_0^{[t]}T_s(h)ds\right\|_\ii\\
&=\left\|\frac1t\int_{[t]}^tT_s(h)ds-\left(\frac1t-\frac1{[t]}\right)\int_0^{[t]}T_s(h)ds\right\|_\ii\\
&\leq2\left(1-\frac{[t]}t\right)\|h\|_\ii\to 0\text{\ \, as\,\ }t\to\ii.
\end{split}
\end{equation*}

Since $\{M_{[t]}(\mc T)(h)\}_{t>1}=\{M_n(\mc T)(h)\}_{n=1}^\ii$ as sequences, it follows that
\begin{equation*}
\begin{split}
\big\|\big(M_t(\mc T)(h)-\wh h\big)\chi_G\big\|_\ii&=\big\|M_t(\mc T)(h)-M_{[t]}(\mc T)(h)\big\|_\ii\\
&+\big\|\big(M_{[t]}(\mc T)(h)-\wh h\big)\chi_G\big\|_\ii\to 0\text{\, \ as\,\ }t\to\ii,
\end{split} 
\end{equation*}
and we conclude $M_t(\mc T)(h)\to\wh h$ a.u. for all $h\in\mc L^1\cap\mc L^\ii$. 

Thus, since the set $\mc L^1\cap\mc L^\ii$ is dense in $\mc L^p$, an application of Corollary \ref{c1} yields that for every $f\in\mc L^p$ there exists $\wh f\in\mc L^0$ such that $M_t(\mc T)(f)\to\wh f$ a.u. Then,  there is a sequence $t_n\to\ii$ for which $M_{t_n}(\mc T)(f)\to\wh f$ a.u. as $n\to\ii$. Hence, as the sequence $\{M_{t_n}(\mc T)(f)\}_{n=1}^\ii$ is bounded in $\mc L^p$, Corollary \ref{c3} implies that $\wh f\in\mc L^p$.
\end{proof}
\vskip 5pt
Next, we shall establish a.u. convergence in the {\it local ergodic theorem}, that is, when $t\to0^+$:
\begin{teo}\label{t4}
Given $1\leq p<\ii$ and $f\in\mc L^p$, the averages $M_t(\mc T)(f)$ converge to $f$ as $t\to0$.
\end{teo}

To prove Theorem \ref{t4}, we will need a local mean ergodic theorem. Let us begin with a lemma:

\begin{lm}\label{l1}
If $t_0>0$, $g\in\mc L^p$, and $f=M_{t_0}(g)$, then $\big\|M_t(\mc T)(f)-f\big\|_p\to 0$ as $t\to0$. In addition, if $g\in\mc L^1\cap\mc L^\ii$, then $\big\|M_t(\mc T)(f)-f\big\|_\ii\to 0$ as $t\to0$.
\end{lm}
\begin{proof}
Let $0<t<t_0$. Then we have
\begin{equation*}
\begin{split}
\big\|M_t(\mc T)(f)&-f\big\|_p=\left\|\frac1t\int_0^t(T_s(f)-f)ds\right\|_p\leq\big\|T_s(f)-f\big\|_p\\
&=\frac1{t_0}\left\|\int_0^{t_0}T_{s+t}(g)ds-\int_0^{t_0}T_s(g)ds\right\|_p\\
&=\frac1{t_0}\left\|\int_t^{t_0+t}T_s(g)ds-\int_0^{t_0}T_s(g)ds\right\|_p\\
&=\frac1{t_0}\left\|\int_t^{t_0}T_s(g)ds+\int_{t_0}^{t+t_0}T_s(g)ds-\int_0^tT_s(g)ds-\int_t^{t_0}T_s(g)ds\right\|_p\\
&=\frac1{t_0}\left\|\int_{t_0}^{t+t_0}T_s(g)ds-\int_0^tT_s(g)ds\right\|_p\leq\frac{2t}{t_0}\|g\|_p\to 0\text{\, \ as\ \,}t\to0.
\end{split}
\end{equation*}

Proof of the second statement is carried out exactly the same way by simply replacing $L^1$-\,norm by $L^\ii$-\,norm.
\end{proof}

Here is the local mean ergodic theorem in $\mc L^p$:
\begin{teo}\label{t5}
If $1\leq p<\ii$ and $f\in\mc L^p$, then $\big\|M_t(\mc T)(f)-f\big\|_p\to 0$ as $t\to0$.
\end{teo}
\begin{proof}
Since, by Proposition \ref{p8}, the semigroup $\{T_s\}_{t\ge 0}$ is strongly continuous in $\mc L^p$, it follows that
\[
\big\|T_s(f)-f\|_p\to 0\text{\, \ as\,\ }s\to0.
\]
Therefore, for every $n\in\mbb N$ there exists $t_n>0$ such that
\[
\big\|T_s(f)-f\big\|_p<\frac1n\ \ \,\forall\, \ 0<s\leq t_0
\]
hence, denoting $f_n=M_{t_n}(f)$, we have
\[
\|f_n-f\|_p=\frac1{t_n}\left\|\int_0^{t_n}(T_s(f)-f)ds\right\|_p<\frac1n.
\]
Fix $\ep>0$ and choose $n_0$ to be such that
\[
\|f_{n_0}-f\|_p<\frac\ep4.
\]
Further, by Lemma \ref{l1}, 
\[
\big\|M_t(\mc T)(f_{n_0})-f_{n_0}\big\|_p\to 0\text{\, \ as\,\ }t\to0,
\]
thus, there exists $t_0>0$ such that
\[
\big\|M_t(\mc T)(f_{n_0})-f_{n_0}\big\|_p<\frac\ep2 \ \ \,\forall\, \ 0<t<t_0. 
\]
Now, we have
\begin{equation*}
\begin{split}
\big\|M_t(\mc T)(f)-f\big\|_p&\leq\big\|M_t(\mc T)(f)-M_t(\mc T)(f_{n_0})\big\|_p+\big\|M_t(\mc T)(f_{n_0})-f_{n_0}\big\|_p\\
&+\|f_{n_0}-f\|_p\leq2\|f_{n_0}-f\|_p+\frac\ep2<\ep
\end{split}
\end{equation*}
whenver $0<t<t_0$, thus completing the argument.
\end{proof}

\begin{proof}[Proof of Theorem \ref{t4}] 
Let a sequence of real numbers $t_n\to 0$. Then, by Theorem \ref{t5}, the set
\[
\mc D=\big\{M_{t_n}(\mc T)(h): \,n\in\mbb N, \, h\in\mc L^1\cap\mc L^\ii\big\}
\]
is dense in $\mc L^1\cap\mc L^\ii$, hence in $\mc L^p$. Besides, by Lemma \ref{l1}, 
\[
\big\|M_{t_n}(g)-g\|_\ii\to 0,\text{\ \,hence\,\ } M_{t_n}(\mc T)(g)\to g\text{ \ a.u.,\ } \ \forall\ g\in\mc D.
\]
Now, as Corollary \ref{c1} implies that the a.u.-convergence set of the averages $\{M_{t_n}(\mc T)\}$ is closed in $\mc L^p$, we conclude that the sequence $\{M_{t_n}(\mc T)(f)\}$ converges a.u. for all $f\in\mc L^p$. Further, given $f\in\mc L^p$, Theorem \ref{t5} yields that $M_{t_n}(\mc T)(f)\to f$ in $\mc L^p$, hence in measure. Therefore, by Proposition \ref{p2}, $M_{t_{n_k}}(\mc T)(f)\to f$ a.u. for some subsequence $\{t_{n_k}\}$, and we conclude that
\[
M_{t_n}(\mc T)(f)\to f\text{ \ a.u. as \ }t_n\to0\ \ \forall\ f\in\mc L^p.
\]

Next, assume that $f\in\mc L^1\cap\mc L^\ii$. Then we have $M_{1/n}(\mc T)(f)\to f$ a.u., in particular, 
\[
M_{[1/t]^{-1}}(f)\to f\text{ \,\ a.u. as\ \,} t\to0.
\] 
So, for a fix $\ep>0$, there exixts $G\su\Om$ such that $\mu(\Om\sm G)\leq\ep$ and
\[
\big\|\big(M_{[1/t]^{-1}}(f)-f\big)\chi_G\big\|_\ii\to 0\text{\, \ as\,\ }t\to0.
\] 
Further,
\begin{equation*}
\begin{split}
\left\|M_t(\mc T)(f)-M_{[1/t]^{-1}}(\mc T)(f)\right\|_\ii&=\left\|\frac1t\int_0^tT_s(f)ds-\left[\frac1t\right]\int_0^{[1/t]^{-1}}T_s(f)ds\right\|_\ii\\
&=\left\|\left(\frac1t-\left[\frac1t\right]\right)\int_0^tT_s(f)ds-\left[\frac1t\right]\int_t^{[1/t]^{-1}}T_s(f)ds\right\|_\ii\\
&\leq2\left(1-\left[\frac1t\right]\, t\right)\|f\|_\ii\to 0\text{ \ as\ }t\to0.\\
\end{split}
\end{equation*}
Therefore, we can write
\begin{equation*}
\begin{split}
\big\|\big(M_t(\mc T)(f)-f\big)\chi_G\big\|_\ii&\leq\big\|M_t(\mc T)(f)-M_{[1/t]^{-1}}(\mc T)(f)\big\|_\ii\\
&+\big\|\big(M_{[1/t]^{-1}}(f)-f\big)\chi_G\big\|_\ii\to 0\text{ \ as\ }t\to0.
\end{split}
\end{equation*}
Thus, the averages $\{M_t(\mc T)\}_{t>0}$ converge a.u. on the set $\mc L^1\cap\mc L^\ii$, which is dense in $\mc L^p$, implying, in view of Corollary \ref{c1}, that the net $\{M_t(\mc T)(f)\}_{t>0}$ converges a.u. as $t\to0$ for all $f\in\mc L^p$. This, in turn, entails that $M_t(\mc T)(f)\to f$ a.u. as $t\to0$, in the same way as in the first part of this proof, as we now know that, given $f\in\mc L^p$, there is a sequence $t_n\to0$ such that $\{M_{t_n}(\mc T)(f)\}$ converges a.u.
\end{proof}

\section{A.u. convergence of the weighted averages in $\mc L^p(\Om,\mu)$}
Let $\mbb C$ be the field of complex numbers, and let $\mbb C_1= \{z\in \mbb C: |z|=1\}$. A function $p:\mbb R_+\to \mbb C$ is called a {\it trigonomertic polynomial} if $p(s)=\sum\limits_{j=1}^nw_j\lb_j^s$, where $n\in \mbb N$, $\{w_j\}_1^n\su \mbb C$, and $\{\lb_j\}_1^n\su \mbb C_1$.

A Lebesgue measurable function $\bt: \mbb R_+\to \mbb C$ will be called {\it bounded Besicovitch} ({\it locally Besicovitch}) {\it function} if  $\|\bt\|_\ii<\ii$, and for every $\ep>0$ there is a trigonometric polynomial $p_\ep$ such that
\begin{equation}\label{e21}
\limsup_{t\to\ii}\frac 1t\int_0^t\big|\bt(s)-p_\ep(s)\big|ds<\ep
\end{equation}
(respectively,
\begin{equation}\label{e21z}
\limsup_{t\to0^+}\frac 1t\int_0^t\big|\bt(s)-p_\ep(s)\big|ds<\ep\, ).
\end{equation}

\vskip 5pt
 Let $\nu$ stand for Lebesgue measure in $\mbb C_1$ and consider the product measure space
\[
(\mbb C_1\otimes\Om, \nu\otimes\mu)=(\mbb C_1, \nu)\otimes(\Om, \mu).
\]
As in \cite[Lemma 4.7]{ccl}), we obtain the following.
\begin{lm}\label{l2}
Let $(X,\nu)$ and $(Y,\mu)$ be $\sg$-finite measure spaces. If a net $\{g_t\}_{t\ge 0}\su\mc L^0(X\otimes Y)$ is such that $g_t\to g$ a.u. in $X\otimes Y$ as $t\to\ii$ or $t\to0^+$, then $g_t(x,y)\to g(x,y)$ a.u. in $Y$ as\, $t\to\ii$ or $t\to0^+$ for almost all $x\in X$.
\end{lm}
Recall that a semigroup $\{T_t\}_{t\ge 0}\su DS$ is presumed to be strongly continuous in $\mc L^1=\mc L^1(\Om)$. Define
\[
f_z(\om)=\wt f(z,\om)\text{ \ whenever \ } \wt f\in\mc L^1(\mbb C_1\otimes\Om)+\mc L^\ii(\mbb C_1\otimes\Om),
\]
and note that if $\wt f\in\mc L^1(\mbb C_1\otimes\Om)$, then, by Fubini's theorem, $f_z\in\mc L^1=\mc L^1(\Om)$ for $\nu$-almost all $z\in\mbb C_1$.

Next, given $\lb \in \mbb C_1$, define
\[
T^{\lb}_s(\wt f)(z,\om)=T_s(f_{\lb^sz})(\om)\text{ \ for\ }s\ge0\text{ \ and \ }\wt f\in\mc L^1(\mbb C_1\otimes\Om)+\mc L^\ii(\mbb C_1\otimes\Om).
\]
Then it is easily verified that $\{T^{(\lb)}_s\}_{s\ge0}\su DS$ on  $\mc L^1(\mbb C_1\otimes\Om)+\mc L^\ii(\mbb C_1\otimes\Om)$ for each $\lb\in\mbb C_1$. Besides, given $r, s\ge0$, we have

\begin{equation*}
\begin{split}
T_{r+s}^{(\lb)}(\wt f)(z,\om)&=T_{r+s}(f_{\lb^{r+s}z})(\om)=T_r(T_s(f_{\lb^s\lb^tz}))(\om)=
T_r(T_s^{(\lb)}(f_{\lb^tz}))(\om)\\
&=T^{(\lb)}_s(T_r(f_{\lb^rz}))(\om)=T_r^{(\lb)}T_s^{(\lb)}(\wt f)(z,\om),
\end{split}
\end{equation*}
that is, $\{T_t^{(\lb)}\}_{t\ge 0}$ is a semigroup.

\begin{pro}\label{p4}
The semigroup  $\{T_s^{(\lb)}\}_{s\ge 0}$ is strongly continuous in $\mc L^1(\mbb C_1\otimes\Om)$ for each $\lb\in\mbb C_1$.
\end{pro}
\begin{proof}
Clearly, it suffices to show that, given $\wt f\in\mc L^1(\mbb C_1\otimes\Om)$ and $\lb\in\mbb C_1$, we have
\[
\big\|T_{s_n}^{(\lb)}(\wt f)-T_s^{(\lb)}(\wt f)\big\|_{\mc L^1(\mbb C_1\otimes\Om)}\to0\text{\, \ as\ \,}s_n\searrow s\text{\, \ and as\, \ }s\nearrow s.
\]
If $s_n\searrow s$ ($s_n\nearrow s$), then, as
\[
\big\|T_{s_n}^{(\lb)}(\wt f)-T_s^{(\lb)}(\wt f)\big\|_{\mc L^1(\mbb C_1\otimes\Om)}\leq\big\|T_{s_n-s}^{(\lb)}(\wt f)-\wt f\big\|_{\mc L^1(\mbb C_1\otimes\Om)}
\]
(respectively, as
\[
\big\|T_{s_n}^{(\lb)}(\wt f)-T_s^{(\lb)}(\wt f)\big\|_{\mc L^1(\mbb C_1\otimes\Om)}\leq\big\|T_{s-s_n}^{(\lb)}(\wt f)-\wt f\big\|_{\mc L^1(\mbb C_1\otimes\Om)}),
\]
we conclude that it is sufficient to verify that
\[
\big\|T_{s_n}^{(\lb)}(\wt f)-\wt f\big\|_{\mc L^1(\mbb C_1\otimes\Om)}=\int_{\mbb C_1}\big\|T_{s_n}(\wt f(\lb^{s_n}z,\cdot))-\wt f(z,\cdot)\big\|_1d\nu(z)\to0\text{\ \, as\ \ }s_n\to0^+.
\]

We have
\begin{equation*}
\begin{split}
\big\|T_{s_n}(\wt f(\lb^{s_n}z,\cdot))&-\wt f(z,\cdot)\big\|_1\leq\big\|T_{s_n}(\wt f(\lb^{s_n}z,\cdot))-T_{s_n}(\wt f(z,\cdot))\big\|_1\\
&+\big\|T_{s_n}(\wt f(z,\cdot))-\wt f(z,\cdot)\big\|_1\\
&\leq\big\|\wt f(\lb^{s_n}z,\cdot)-\wt f(z,\cdot)\big\|_1+\big\|T_{s_n}(\wt f(z,\cdot))-\wt f(z,\cdot)\big\|_1.
\end{split}
\end{equation*}
Since $\{T_s\}_{s\ge0}$ is strongly continuous in $\mc L^1$, it follows that
\[
\big\|T_{s_n}(\wt f(z,\cdot))-\wt f(z,\cdot)\big\|_1\text{ \ for $\nu$-almost all \ }z\in\mbb C_1,
\]
which, together with $\big\|T_{s_n}(\wt f(z,\cdot))-\wt f(z,\cdot)\big\|_1\leq2\|\wt f(z,\cdot)\|_1$\, $\nu$-a.e. on $\mbb C_1$, implies that
\[\
\int_{\mbb C_1}\big\|T_{s_n}(\wt f(z,\cdot))-\wt f(z,\cdot)\big\|_1d\nu(z)\to0\text{ \ as\ }s_n\to0^+.
\]

Thus, it remains to verify that
\begin{equation}\label{e5}
\int_{\mbb C_1}\big\|\wt f(\lb^{s_n}z,\cdot)-\wt f(z,\cdot)\big\|_1d\nu(z)\to0\text{ \ as\ }s_n\to0^+.
\end{equation}
To that end, let
\begin{equation}\label{e6}
\wt h(z,\om)=\sum_{i=1}^mh_i\chi_{[\lb_1,\lb_{i+1}]}(z),\, \ z\in\mbb C_1, \ h_i\in\mc L^1\text{ \ for\,\ }i=1,\dots,m
\end{equation}
be a simple Bochner-measurable function, where $\{\lb_i\}_{i=1}^{m+1}$, with $\lb_{m+1}=\lb_1$, is a {\it partition} of $\mbb C_1$. Then, given $s\in\mbb R$, we have
\[
\wt h(\lb^sz,\om)=\sum_{i=1}^mh_i\chi_{[\lb^s\lb_1,\lb^s\lb_{i+1}]}(z),
\]
where $\{\lb^s\lb_i\}_{i=1}^{m+1}$  is another partition of $\mbb C_1$. Since $\lb^{s_n}\lb_i\to\lb_i$ as $s_n\to0$ for each $i$, it follows that
\[
\int_{\mbb C_1}\big\|\wt h(\lb^{s_n}z,\cdot)-\wt h(z,\cdot)\big\|_1d\nu(z)\leq\sum_{i=1}^m\big|\chi_{[\lb^{s_n}\lb_i,\lb^{s_n}\lb_{i+1}]}-\chi_{[\lb_i,\lb_{i+1}]}\big|\,\|h_i\|_1\to0
\]
as $s_n\to0^+$.

Now, as $\wt f\in\mc L^1(\mbb C_1\otimes\Om)$, one can justify that there exists a sequence $\{\wt h_k(z,\om)\}$ of functions of the form (\ref{e6}) for which
\[
\int_{\mbb C_1}\big\|\wt f(z,\cdot)-\wt h_k(z,\cdot)\big\|_1d\nu(z)\to0\text{\, \ as\,\ }k\to\ii.
\]
Then we have
\begin{equation*}
\begin{split}
\int_{\mbb C_1}\big\|\wt f(\lb^{s_n}z,\cdot)&-\wt f(z,\cdot)\big\|_1d\nu(z)\leq\int_{\mbb C_1}\big\|\wt f(\lb^{s_n}z,\cdot)-\wt h_k(\lb^{s_n}z,\cdot)\big\|_1d\nu(z)\\
&+\int_{\mbb C_1}\big\|\wt h_k(\lb^{s_n}z,\cdot)-\wt h_k(z,\cdot)\big\|_1d\nu(z)\\
&+\int_{\mbb C_1}\big\|\wt h_k(z,\cdot)-\wt f(z,\cdot)\big\|_1d\nu(z)=2\int_{\mbb C_1}\big\|\wt h_k(z,\cdot)-\wt f(z,\cdot)\big\|_1d\nu(z)\\
&+\int_{\mbb C_1}\big\|\wt h_k(\lb^{s_n}z,\cdot)-\wt h_k(z,\cdot)\big\|_1d\nu(z)\to0\text{\, \ as\, \ } s_n\to0^+\ \ \forall \ k,
\end{split}
\end{equation*}
and (\ref{e5}) follows, completing the proof.
\end{proof}

\begin{lm}\label{l3}
Let a semigroup $\mc T=\{T_s\}_{s\ge 0}\su DS$ be strongly continuous in $\mc L^1$, and let $p(s)=\sum\limits_{j=1}^nw_j\lb_j^s$ be a trigonometric polynomial. If $1\leq p<\ii$, $f\in\mc L^p$, and
\begin{equation}\label{e7}
P_t(\mc T)(f)=\frac 1t\int_0^tp(s)T_s(f)ds,
\end{equation}
then
\begin{enumerate}[(i)]
\item
the net $P_t(\mc T)(f)$ converges a.u. as $t\to \ii$;
\item
the net $P_t(\mc T)(f)$ convergs a.u. to $p(0)f$ as $t\to0$.
\end{enumerate}
\end{lm}
\begin{proof}

(i) Fix  $\lb \in \mbb C_1$. In view of Propositions \ref{p8} and \ref{p4}, $\{T_s^{(\lb)}\}_{s\ge 0}\su DS$ is a semigroup strongly continuous in $\mc L^p(\mbb C_1\otimes\Om)$. Then, by Theorem \ref{t3}, given $\wt f\in\mc L^p(\mbb C_1\otimes\Om)$, the net
\begin{equation}\label{e0}
\frac 1t\int_0^tT_s^{(\lb)}(\wt f)ds
\end{equation}
converges a.u. as $t\to \ii$. Therefore, by Lemma \ref{l2}, the net
\[
\frac 1t\int_0^tT_s^{(\lb)}(\wt f)(z,\om)ds=\frac 1t\int_0^tT_s(\wt f(\lb^sz,\om))ds
\]
converges a.u. as $t\to \ii$ \,$\nu$-a.e. on $\mbb C_1$. In particular, letting $\wt f(z,\om)=zf(\om)$, we conclude
that the net
\[
z\,\frac 1t\int_0^t\lb^sT_s(f)ds
\]
converges a.u. as $t\to\ii$ for some $0\neq z\in \mbb C_1$, implying that the net
\[
\frac 1t\int_0^t\lb^sT_s(x)ds
\]
converges a.u. as $t\to\ii$. Therefore, by linearity, the net $P_t(\mc T)(x)$ converges a.u. as $t\to\ii$.

(ii) Now, by Theorem \ref{t5}, given $\wt f\in\mc L^p(\mbb C_1\otimes\Om)$, it follows that the net (\ref{e0})
converges a.u. to $\wt f$ as $t\to0$. Then, letting $\wt f(z,\om)=zf(\om)$, we see as above that
\[
\frac 1t\int_0^t\lb^sT_s(f)ds\to f \text{ \ a.u.}
\]
as $t\to0$, and the result follows by linearity.
\end{proof}

Now we are ready to present the main results of this section.
\begin{teo}\label{t6}
Let a semigroup $\{T_s\}_{s\ge 0}\su DS$ be strongly continuous in $\mc L^1$, and let $\bt: \mbb R_+\to\mbb C$ be a bounded Besicovitch function with $\|\bt\|_\ii<C<\ii$. Then, given $1\leq p<\ii$ and $f\in\mc L^p$, the net $\{M_t(\mc T,\bt)(f)\}_{t>0}$ converges a.u.  as $t\to \ii$ to some $\wh f\in\mc L^p$.
\end{teo}
\begin{proof}
Assume that $f\in\mc L^1\cap \mc L^\ii$. Fix $\ep>0$ and choose a trigonometric polynomial $p=p_\ep$ to satisfy condition (\ref{e21}). Let $\{P_t(x)\}_{t>0}$ be given by (\ref{e7}). Then we have
\begin{equation}\label{e8}
\big\| M_t(\mc T,\bt)(f)-P_t(\mc T)(f)\big\|_\ii\leq\frac 1t\int_0^t\big|\bt(s)-p(s)\big|ds\, \|f\|_\ii<\ep\,\|f\|_\ii
\end{equation}
for all big enough values of $t$. Since, by Lemma \ref{l3}, the averages $P_t(\mc T)(f)$ converge a.u.,
it follows that the net $\{M_t(\mc T,\bt)(f)\}_{t>0}$ is a.u. Cauchy as $t\to\ii$. Therefore, by Proposition \ref{p0}, this net converges to some $\wh f\in\mc L^0$, and as in the end of proof of Theorem \ref{t3}, it follows that $\wh f\in\mc L^p$.
\end{proof}

\begin{teo}\label{t7}
Let a semigroup $\{T_t\}_{t\ge 0}\su DS$ be strongly continuous in $\mc L^1$, and let $\bt: \mbb R_+\to\mbb C$  be a bounded locally Besicovitch function with $\|\bt\|_\ii<C<\ii$. Then, given $1\leq p<\ii$ and $f\in\mc L^p$, the averages $\{M_t(\mc T,\bt)(f)\}_{t>0}$ converge a.u. to $\lb(f)\,f$ as $t\to0$ for some $\lb(f) \in \mbb C$.
\end{teo}
\begin{proof}
Assume that $f\in\mc L^1\cap\mc L^\ii$. Fix $\ep>0$ and choose a trigonometric polynomial $p=p_\ep$ to satisfy condition (\ref{e21z}). If the net $\{P_t(\mc T)(f)\}_{t>0}$ is given by (\ref{e7}), then (\ref{e8}) holds for all small enough positive values of $t$. Since, by Lemma \ref{l3}, the averages $P_t(\mc T)(f)$ converge a.u. as $t\to0$, it follows that the net $\{M_t(\mc T,\bt)(f)\}_{t>0}$ is a.u. Cauchy as $t\to0$. Then, as in Theorem \ref{t6}, we conclude that for any $f \in\mc L^p$ the averages $\{M_t(\mc T,\bt)(f)\}_{t>0}$ converge a.u. to some $\wh f\in \mc L^p$ as $t\to0$.

Let $p_n$ be a trigonometric polynomial to satisfy condition (\ref{e21z}) with $\ep=1/n$, and let $\{P^{(n)}_t(\mc T)(f)\}_{t>0}$ be the corresponding averages given by (\ref{e7}). Then there exists $t_n>0$ such that
\begin{equation}\label{e4}
\big\| M_t(\mc T,\bt)(f)-P^{(n)}_t(\mc T)(f)\big\|_p\leq\frac 1t\int_0^t\big|\bt(s)-p_n(s)\big|ds\, \|f\|_p<\frac{\|f\|_p}n
\end{equation}
holds for all  $0<t<t_n$.

By Lemma \ref{l3}, $P^{(n)}_t(\mc T)(f)\to p_n(0)\,f$ a.u. as $t\to0$, implying that there exists a sequence $t_m\to0$ such that
\[
M_{t_m}(\mc T,\bt)(f) -P^{(n)}_{t_m}(\mc T)(f)\to\wh f-p_n(0)\,f \text{\ \ a.u. as \ } m\to\ii.
\]
Since, by Corollary \ref{c3}, the unit ball in $\mc L^p$ is sequentially closed with respect to a.u. convergence, (\ref{e4}) entails that
\[
\big\|\wh f - p_n(0)\,f\big\|_p\leq\frac{\|f\|_p}n,
\]
hence $\wh f=\|\cdot\|_p-\lim\limits_{n\to\ii}p_n(0)\,f$, and we conclude that $\wh f=\lb(f)\,f$ for some $\lb(f)\in\mbb C$, which completes the argument.
\end{proof}

\section{Extension to fully symmetric spaces}

Now we aim to extend the results of Section 5 to the fully symmetric spaces $E\su\mc L^1+\mc L^\ii$ with
$\mb 1=\chi_\Om\notin E$.

As before, a semigroup $\mc T=\{T_s\}_{s\ge0}\su DS$ is assumed to be strongly continuous in $\mc L^1$. Then, by Proposition \ref{p1}, if $\bt: \mbb R_+\to\mbb C$ is a Lebesgue measurable function such that $\|\bt\|_\ii<C<\ii$, the averages $M_t(\mc T,\bt)(f)$ are defined for every $f\in\mc L^1$ and $t>0$. Furthemore, as the operators $C^{-1}\,M_t(\mc T,\bt)$ contract both $\mc L^1$ and $\mc L^1\cap\mc L^\ii$, we may assume,  in view of Theorem \ref{t0}, that the averages $M_t(\mc T,\bt)(f)$ are defined for all $f\in\mc L^1+\mc L^\ii$. Note that then $C^{-1}M_t(\mc T,\bt)\in DS$ for all $t>0$.

\vskip5pt
Denote
\[
\mc R_\mu=\left\{f \in\mc L^1+\mc L^\ii: \ \mu\{|f|> \lb\}<\ii \text{ \ for all \ } \lb>0\right\},
\]
the Fava's space -- see \cite[Section (2.5.5)]{es}. (One can observe that $\bigcup\limits_{1\leq p<\ii}\mc L^p \subsetneqq\mc R_\mu$.) 

As it was shown -- see \cite{cl18,cl19,ku} -- that the pointwise Dunford-Schwartz ergodic theorem holds for $f\in\mc L^1+\mc L^\ii$ if and only if $f\in\mc R_\mu$, we will only be interested in the functions that belong to the Fava's space $\mc R_\mu$.

\vskip5pt
We will utilize the following characterization of $\mc R_\mu$ -- see \cite[Proposition 2.1]{ccl}:

\begin{pro}\label{p5}
Let $f\in\mc L^1+\mc L^\ii$. Then $f\in\mc R_\mu$ if and only if for each $\ep>0$ there exist $g_\ep\in \mc L^1$  and $h_\ep\in\mc L^\ii$ such that
\[
f=g_\ep+h_\ep\text{ \ \  and \ \ }\| h_\ep\|_\ii\leq\ep.
\]
\end{pro}

\begin{cor}\label{c2}
If $T\in DS$, then $T(\mc R_\mu)\su\mc R_\mu$, hence $M_t(\mc T,\bt)(\mc R_\mu)\su\mc R_\mu$ for every Lebesgue measurable function $\bt: \mbb R_+\to\mbb C$ such that $\|\bt\|_\ii<\ii$ and $t>0$.
\end{cor}

Let us also note the following -- see \cite[Theorem 2.3]{li1}, where the proof can be straightforwardly adapted to the current commutative setting:

\begin{pro}\label{p6}
The space $\mc R_\mu$ is complete with respect to a.u. convergence.
\end{pro}

\begin{teo}\label{t8}
If a semigroup $\mc T=\{T_s\}_{s\ge0}\su DS$ is strongly continuous in $\mc L^1$ and $\bt$ is a bounded Besicovitch (locally Besicovitch) function, then for every $f\in\mc R_\mu$ the averages $M_t(\mc T,\bt)(f)$ converge to some $\wh f\in\mc R_\mu$ as $t\to\ii$ (respectively, as $t\to0$).
\end{teo}
\begin{proof}
Assume that $\bt$ is a bounded Besicovitch function with $\|\bt\|_\ii<C<\ii$. Fix $\ep>0$ and $\dt>0$. By Proposition \ref{p5}, there exist $g\in\mc L^1$ and $h\in\mc L^\ii$ such that
\[
f=g+h\text{\, \ and\, \ } \|h\|_\ii\leq\frac\dt{3C}.
\]
Since $g\in\mc L^1$, Theorem \ref{t6} implies that there exists $G\su\Om$ and $t_0>0$ such that
\[
\mu(\Om\sm G)\leq\ep\text{\, \ and\, \ } \big\|(M_{t^\pr}(\mc T,\bt)(g)-M_{t^{\pr\pr}}(\mc T,\bt)(g))\chi_G\big\|_\ii\leq\frac\dt3\, \ \ \forall\ t^\pr, t^{\pr\pr}\ge t_0.
\]
Then, given $t^\pr, t^{\pr\pr}\ge t_0$, we have
\begin{equation*}
\begin{split}
\big\|(M_{t^\pr}(\mc T,\bt)(f)&-M_{t^{\pr\pr}}(\mc T,\bt)(f))\chi_G\big\|_\ii\leq\big\|(M_{t^\pr}(\mc T,\bt)(g)-M_{t^{\pr\pr}}(\mc T,\bt)(g))\chi_G\big\|_\ii\\
&+\big\|(M_{t^\pr}(\mc T,\bt)(h)-M_{t^{\pr\pr}}(\mc T,\bt)(h))\chi_G\big\|_\ii\\
&\leq\frac\dt3+\big\|M_{t^\pr}(\mc T,\bt)(h)\big\|_\ii+\big\|M_{t^{\pr\pr}}(\mc T,\bt)(h)\big\|_\ii\leq\frac\dt3+2C\,\|h\|_\ii\leq\dt.
\end{split}
\end{equation*}
Therefore, the net $\{M_t(\mc T,\bt)(f)\}_{t>0}\su\mc R_\mu$ is a.u. Cauchy, implying, by Proposition \ref{p6}, that it converges a.u. to some $\wh f\in\mc R_\mu$.

Next, let us assume that $\bt$ is a bounded locally Besicovitch function. Then, repeating the above argument with Theorem \ref{t7} instead of Theorem \ref{t6}, we conclude that the net $\{M_t(\mc T,\bt)(f)\}_{t>0}$ converges a.u. to some $\wh f\in\mc R_\mu$ as $t\to0$.
\end{proof}

Define
\[
\mc L_\mu^0=\big\{ f\in\mc L^0:\, \mu\{|f|>\lb\}<\ii\text{ \ for some \ }\lb>0\big\}.
\]
Given $f\in\mc L_\mu^0$, the {\it non-increasing rearrangement} of $f$ is defined as
\[
f^*(t)=\inf\big\{ \lb>0:\, \mu\{|f|>\lb\}<t\big\}, \ \, t>0.
\]
A non-zero linear subspace $E\su\mc L_\mu^0$ with a Banach norm $\|\cdot\|_E$ is called {\it symmetric} ({\it fully symmetric}) if $g\in E$, $f\in\mc L_\mu^0$, and
\[
f^*(t)\leq g^*(t)  \ \big(\text{respectively,\ } \int_0^sf^*(t)dt\leq\int_0^sg^*(t)dt\text{\ (in short,\ } f\prec\prec g) \big)\, \ \forall\ s>0
\]
imply that $f\in E$ and $\|f\|_E\leq\|g\|_E$.

\vskip5pt
Immediate examples of fully symmetric spaces are the spaces $\mc L^p=\mc L^p(\Om.\mu)$, $1\leq p\leq\ii$, with the standard norms $\|\cdot\|_p$. Further examples include the well-known Orlicz, Lorentz, and Marcinkiewicz spaces -- for details, see, for example, \cite[Sec.\,5]{cl19},  \cite[Sec.\,6]{ccl}.

\vskip5pt
Here is a criterion for a symmetric space to belong to the space $\mc R_\mu$ -- see \cite[Proposition 6.1]{ccl}:

\begin{pro}\label{p7}
Assume that $\mu(\Om)=\ii$. Then a symmetric space $E$ is contained in $\mc R_\mu$ if and only if $\mb 1\notin E$.
\end{pro}

Let us now present an application of Theorem \ref{t8} to fully symmetric spaces:

\begin{teo}\label{t42}
Let $E$ be a fully symmetric space such that $\mb 1\notin E$, and let $\mc T$ and $\bt(s)$ be as in Theorem \ref{t8}. Then, given $f\in E$, the averages $M_t(\mc T,\bt)(f)$ converge a.u. to some $\wh f\in E$ as $t\to\ii$ (respectively, as $t\to0$).
\end{teo}
\begin{proof}
Since $\mb 1\notin E$, it follows that $E\su\mc R_\tau$. Therefore, by Theorem \ref{t8}, the averages $\{M_t(\mc T,\bt)(f)\}_{t>0}$ converge a.u. to some $\wh f\in\mc R_\tau$ as $t \to \ii$ (respectively, as $t\to0$),  so, it remains to show that $\wh f\in E$, for which we will follow the argument of that of \cite[Theorem 6.2]{ccl},

Let $\{t_n\}$ be a sequence such that the averages $\{M_{t_n}(\mc T,\bt)(f)\}_{n=1}^\ii$ converge a.u. when $t_n\to\ii$ (respectively, $t_n\to0$). If $C$ is such that $\|\bt\|_\ii<C$, then $C^{-1}M_{t_n}(\mc T,\bt)\in DS$ for every $n$, implying, by \cite[Ch.\,II, \S\,3, Sec.\,4]{kps}, that
\begin{equation}\label{e9}
C^{-1}M_{t_n}(\mc T,\bt)(f)\prec\prec f\, \ \ \forall\ n.
\end{equation}
Next, as $M_{t_n}(\mc T,\bt)(f)\to\wh f$ also in measure, we observe, by \cite[Ch.\,II, \S\,2, $11^0$]{kps}, that 
\[
\big(C^{-1}M_{t_n}(\mc T,\bt)(f)\big)^*\to\big(C^{-1}\wh f\,\big)^*\text{ \ \ a.e. on \ } (0,s)\text{ \ as \ } n\to\ii\ \ \forall \ s>0.
\]
Therefore, it follows from Fatou's lemma and (\ref{e9}) that
\[
\int_0^s\big(C^{-1}\wh f\,\big)^*dt\leq\liminf_n\int_0^s\big(C^{-1}M_{t_n}(\mc T,\bt)(f)\big)^*dt\leq\int_0^sf^*dt\ \ \ \forall \ s>0,
\]
hence $(\wh f\,)^*\prec\prec C\,f^*$. Since $E$ is a fully symmetric space and $f\in E$, we conclude that $\wh f\in E$.
\end{proof}

\begin{rem}
One may observe that all the above arguments remain applicable if the net $\{t\}_{t>0}$ is replaced by an arbitrary net $\{t_\al\}_{\al\in\Lb}\su(0,\ii)$ with the directed set $\Lb$ such that $\displaystyle\lim_{\al\in\Lb}t_\al=\ii$, or $\displaystyle\lim_{\al\in\Lb}t_\al=0$.
\end{rem}

\vskip 5pt
\noindent
{\bf Acknowledgment}. The author would like to express his gratitude to Professor Vladimir Chilin for the invaluable learning experinces he had generously created over many years of our collaboration. This work was motivated by these experiences and partially supported by the Pennsylvania State University traveling program.


\begin{thebibliography}{99}

\bibitem{be} A. S. Besicovitch, {\it Almost Periodic Functions}, Cambridge Univ. Press, 1932.

\bibitem{bo} J. R. Baxter, J. H. Olsen, {\it Weighted and subsequential ergodic theorems}, Canad. J. Math. 35 (1983), 145-166.

\bibitem{bl} A. Bellow, V. Lozert, {\it The weighted pointwise ergodic theorem and the individual ergodic theorem along sequences}, Trans. Amer. Math. Soc. 288\,(1) (1985), 307-345.

\bibitem{cl18} V. Chilin, S. Litvinov, {\it The validity space of Dunford–Schwartz pointwise ergodic theorem}, J. Math. Anal. Appl. 461 (2018), 234–247.

\bibitem{cl19a} V. Chilin, S. Litvinov, {\it Local ergodic theorems in symmetric spaces of measurable operators}, Integr. Equ. Oper. Theory (2019), 91:15, 16 pages.

\bibitem{cl19} V. Chilin, S. Litvinov, {\it Almost uniform and strong convergences in ergodic theorems for symmetric spaces}, Acta Math. Hungar. 157\,(1) (2019), 229-253.

\bibitem{cl20} V. Chilin, S. Litvinov, {\it Noncommutative weighted individual ergodic theorems with continuous time}, Inf. Dimen. Anal. Quant. Prob. Rel. Topics {\bf 23}\,(2) (2020), 20 pages.

\bibitem{ccl} V. Chilin, D. \c C\" omez, and S. Litvinov, {\it Individual ergodic theorems for infinite measure}, Colloq. Math. 167\,(2) (2022), 219-238.

\bibitem{ds} N. Dunford, J. T. Schwartz, {\it Linear Operators, Part I: General Theory}, Wiley, 1988.

\bibitem{dj} C. Demeter, R. L. Jones, {\it Besicovitch weights and the necessity of duality restrictions in the weighted ergodic theorem}, Chapel Hill Ergodic Theory Workshops, Contemp. Math.  356 (2004), Amer. Math. Soc., Providence, RI, 127-135.

\bibitem{es} G. A. Edgar, L. Sucheston, {\it Stopping Times and Directed Processes}, Cambridge University Press, 1992.

\bibitem{ga} A. M. Garsia, {\it Topics in Almost Everywhere Convergence}, Markham, 1970.

\bibitem{kps} S. G. Krein, Ju. I. Petunin, and E. M. Semenov, {\it Interpolation of Linear Operators}, Translations of Mathematical Monographs, Amer. Math. Soc. 54, 1982.

\bibitem{kr} U. Krengel, {\it Ergodic Theorems}, de Gruyter, 1985.

\bibitem{ku} D. Kunszenti-Kovács, {\it Counter-examples to the Dunford–Schwartz pointwise ergodic
theorem on $L^1 + L^\ii$}, Arch. Math. (Basel) 112 (2019), 205–212.

\bibitem{lot} M. Lin, J. Olsen, and A. Tempelman, {\it On modulated ergodic theorems for Dunford-Schwartz operators}, Illinois J. Math. 43\,(3) (1999), 542-567.

\bibitem{li1} S. Litvinov, {\it Notes on noncommutative ergodic theorems}, Proc. Amer. Math. Sci. 150\,(8) (2024), 3381-3391.

\bibitem{mg} S. A. McGrath, {\it Local ergodic theorems for noncommutiong semigroups}, Proc. Amer. Math, Sci. 79\,(2) (1980), 212-216.

\bibitem{mk} F. Mukhamedov and A. Karimov, {\it On noncommutative weighted local ergodic theorems
on $L^p$-spaces}, {\it Period. Math. Hungar.} 55 (2007), 223–235.

\bibitem{or} E. T. Ordman, {\it Classroom Notes: Convergence almost everywhere is not topological}, Amer.
Math. Monthly 73\,(2) (1966), 182–183.

\bibitem{ry} C. Ryll-Nardzewski, {\it Topics in ergodic theory}, Lecture Notes in Math. 472 (1975), Springer-Verlag, 131-156.

\bibitem{sa} R. Sato, {\it Ergodic theorems for $d$-parameter semigroups of Dunford-Schwartz operators}, Math. J. Okayama Univ. 23 (1981), 41-57.

\bibitem{te} T. R. Terrell, {\it The local ergodic theorem and semigroups of non-positive operators},  J. Funct. Anal. 10 (1972), 424-429.

\bibitem{yo} K. Yosida, {\it Functional Analysis}, Springer Verlag, Berlin-G\"ottingen-Heidelberg, 1965.
\end{thebibliography}
\end{document}